\documentclass[12pt]{amsart} 

\usepackage[T2A]{fontenc}
\usepackage[cp1251]{inputenc}

\newtheorem{theorem}{Theorem}[section]

\newtheorem{lemma}{Lemma}[section]

\newtheorem{example}{Example}[section]

\usepackage{geometry}
\usepackage{amssymb}
\usepackage{amsmath}
\usepackage{amsthm}
\usepackage{mathtools}
\usepackage{multirow}
\usepackage{thmtools}

\usepackage{tikz}
\usetikzlibrary{calc}
\usetikzlibrary{braids}

\usepackage[colorlinks, citecolor=blue,  backref=page]{hyperref}

\begin{document}

\title[Virtual braids and cluster algebras]{Virtual braids and cluster algebras}
\author[A.~Egorov]{Andrey Egorov}
\address{Tomsk State University, 634050 Tomsk, Russia; Sobolev Institute of Mathematics, 630090 Novosibirsk, Russia; 
Novosibirsk State University, 630090 Novosibirsk, Russia}
\email{a.egorov2@g.nsu.ru}

\begin{abstract}
In 2015 Hikami and Inoue constructed a representation of the braid group $B_n$ in terms of cluster algebra associated with the decomposition of the complement of the corresponding knot into ideal hyperbolic tetrahedra. This representation leads to the calculation of the hyperbolic volume of the complement of the knot that is the closure of the corresponding braid. In this paper, based on the Hikami–Inoue representation discussed above, we construct a representation for the virtual braid group $VB_n$. We show that the so-called “forbidden relations” do not hold in the image of the resulting representation. In addition, based on the developed method, we construct representations for the flat braid group $FB_n$ and the flat virtual braid group $FVB_n$.
\end{abstract} 

\thanks{The author wish to thank Andrei Vesnin for helpful comments and discussions.} 
\thanks{The author was supported by the Tomsk State University Development Program (Priority-2030) and the article was prepared within the framework of the project "Mirror Laboratories" HSE University, RF}


\footnotesize

\maketitle


\section{Introduction} 

Let us start with recalling braid groups and related groups. For $n \geq 2$, the \emph{braid group} $B_n$ is defined as a group with generators $\sigma_1, \ldots, \sigma_{n-1}$ and the following defining relations~\cite{Ar47}: 
\begin{eqnarray}	
\sigma_i \sigma_{i+1} \sigma_i & = & \sigma_{i+1} \sigma_i \sigma_{i+1},  \qquad  i=1, 2, \ldots, n-2, \label{rel1}\\
\sigma_i \sigma_j  & = & \sigma_j \sigma_i,  \qquad \qquad  |i - j| \geq 2.	\label{rel2}	
\end{eqnarray}
A geometric interpretation of $B_n$ is well known, it is isomorphic to a group of geometric braids on $n$ strings, and a mapping class group of an $n$-punctured disc~\cite{KaTu}. 
By adding the relations
\begin{equation}	
\sigma_i^2 = 1,  \qquad  i=1, 2, \ldots, n-1, \label{relflat}
\end{equation}
we get the \emph{flat braid group} $FB_n$ on $n$ strings. 	

The \emph{virtual braid group} $VB_n$ on $n$ strings is the group with two families of generators, classical and virtual, denoted by 
$\sigma_1, \ldots, \sigma_{n-1}$ and $\rho_1, \ldots, \rho_{n-1}$,  with the following defining relations: (\ref{rel1}) and (\ref{rel2}) for classical generators; (\ref{rel3}), (\ref{rel4}) and (\ref{rel5}) for virtual generators,   
\begin{eqnarray}
\rho_i \rho_{i+1} \rho_i  & = & \rho_{i+1} \rho_i \rho_{i+1}, \qquad  i=1, 2, \ldots, n-2, \label{rel3} \\
\rho_i \rho_j &  = & \rho_j \rho_i, \qquad \qquad \quad |i - j| \geq 2,  \label{rel4} \\
\rho_i^2  & = & 1,  \qquad \qquad  \qquad  i=1, 2, \ldots, n-1, \label{rel5}
\end{eqnarray} 
and mixed relations~(\ref{rel6}) and (\ref{rel7}) for both classical and virtual generators
\begin{eqnarray}
\sigma_i \rho_j  & = & \rho_j \sigma_i, \qquad \qquad |i - j| \geq 2, \label{rel6} \\
\rho_i \rho_{i+1} \sigma_i  & = & \sigma_{i+1} \rho_i \rho_{i+1}, \qquad  i=1, 2, \ldots, n-2. \label{rel7}
\end{eqnarray} 
It was observed in~\cite{GPV00} that relations~(\ref{rel8}) and (\ref{rel9}) 
\begin{eqnarray}
\rho_i \sigma_{i+1} \sigma_i & = & \sigma_{i+1} \sigma_i \rho_{i+1}, \label{rel8} \\ 
\rho_{i+1} \sigma_i \sigma_{i+1} & = & \sigma_i \sigma_{i+1} \rho_i \label{rel9}
\end{eqnarray}
do not hold in~$VB_n$, so these relations are called~\emph{forbidden} relations. 
By adding relation (\ref{relflat}) to $VB_n$ we get the \emph{flat virtual braid group} $FVB_n$ on $n$ strings. 

Described above relation between braid groups and virtual braid groups admits to construct representations of $VB_n$ by extending  known representations of $B_n$ by corresponding to $\rho_i$ suitable involutions. In particular, Bardakov, Vesnin and Wiest~\cite{BVW12} constructed a representation of $VB_n$ by extending Dynnikov representa\-tion~\cite{D02}, and demonstrated that the representation from~\cite{BVW12} is faithful for $n=2$ and distinguish virtual braids on three strings good enough. Gotin~\cite{Go17} constructed a representation of $VB_n$ by extending a representation of $B_n$ through rook algebras given by Bigelow, Ramos and Yi~\cite{BRY11}. 

In the present note we construct a representation of $VB_n$ by extending a representation of $B_n$ given by Hikami and Inoue in~\cite{HI15} in terms of a cluster algebra (Theorem~\ref{th1}). . It was demonstrated in~\cite{CYZ18} the the representation from~\cite{HI15} allows to compute volume of hyperbolic knot which is the closer of a braid. Further we also construct representations for flat braid group and virtual flat braid groups (Theorem~\ref{th3} and Theorem~\ref{th4}).

\section{Cluster nutations}

Let $V$ be a complex vector space. An automorphism $R$ of the tensor product $V \otimes V$ is said to be an \emph{R-operator} if it satisfies the following  \emph{Yang~- Baxter equation}
$$
(R \otimes \operatorname{Id}) (\operatorname{Id} \otimes R) (R \otimes \operatorname{Id}) = ( \operatorname{Id} \otimes R) (R \otimes \operatorname{Id}) (\operatorname{Id} \otimes R),
$$ 		
where $\operatorname{Id}$ is the identity operator $\operatorname{Id} : V \to V$.

Let us recall the constuction of $R$-operator from~\cite{HI15}. Denote by $\mathbb F_N$ the field of rational functions over $\mathbb C$ of $N$ algebraically independent variables ${\bf x} =(x_1,...,x_N)$. A \emph{cluster seed} is a pair $({\bf x}, {\bf B})$, where
\begin{itemize}
\item ${\bf x} =(x_1,...,x_N)$ is an ordered set of $N$ algebraically independent variables,  
\item ${\bf B}=(b_{ij})$ is an antisymmetric $N \times N$ -- matrix of integers.  
\end{itemize}

For any $k =1, \ldots, N$ define a \emph{mutation} $\mu_k$ of a seed $({\bf x}, {\bf B})$ in direction $k$ as follows  
$$
\mu_k ({\bf x}, {\bf B}) = ({\bf \tilde{x}}, {\bf \tilde{B}}),
$$
where ${\bf \tilde{x}} = (\tilde{x}_1, \ldots, \tilde{x}_N)$ is defined by the rule 
\begin{equation} 
\tilde{x}_i = \begin{cases}
x_i, & \text{if } i \neq k, \\
\frac{\displaystyle 1}{\displaystyle x_k} \left( \prod_{j: b_{jk}>0} \, x_{j}^{b_{jk}} +  \prod_{j: b_{jk}< 0} \, x_{j}^{-b_{jk}} \right), & \text{if } i=k,
\end{cases} \label{x-mute}
\end{equation}
and matrix  ${\bf \tilde{B}} = (\tilde{b}_{ij})$ is calculating by the formula: 
\begin{equation}
\tilde{b}_{ij} = \begin{cases}
-b_{ij}, & \text{if } i=k \text{ or} j=k, \\
 b_{ij} + \frac{\displaystyle |b_{ik}| b_{kj} + b_{ik} |b_{kj}|}{\displaystyle 2}, & \text{ otherwise.}
\end{cases} \label{B-mute}
\end{equation}
A pair $({\bf \tilde{x}}, {\bf \tilde{B}})$ is a cluster seed again. 

Using cluster variables $\bf x$ we define cluster variables ${\bf y} = (y_1, \ldots, y_N)$ by setting  
\begin{equation}
y_j = \prod_{k=1}^N x_k^{b_{kj}}. \label{y-def}
\end{equation}
Mutation $\mu_k$ induces a mutation of a pair  $({\bf y}, {\bf B})$,  $\mu_k ({\bf y}, {\bf B}) = ({\bf \tilde{y}}, {\bf \tilde{B}})$, where  
${\bf \tilde{B}}$ is given by formula~(\ref{B-mute}) and ${\bf \tilde{y}} = (\tilde{y}_1, \ldots, \tilde{y}_N)$ is given by the following formulas: 
\begin{equation}
\tilde{y}_i = \begin{cases}
y_k^{-1}, & \text{if } i=k, \\
y_i (1 + y_k^{-1})^{-b_{ki}}, & \text{if } i \neq k \text{ and } b_{ki} \geq 0, \\
y_i (1 + y_k)^{-b_{ki}}, & \text{if } i \neq k \text{ and }  b_{ki} \leq 0. 
\end{cases} \label{y-mute}
\end{equation}

In~\cite{HI15} a matrix $\bf B$ was taken equal to the adjacency matrix of a quiver (oriented graph) $\Gamma$ presented in figure~\ref{quiver}. Graph $\Gamma$ has $N=3n+1$ vertices. 
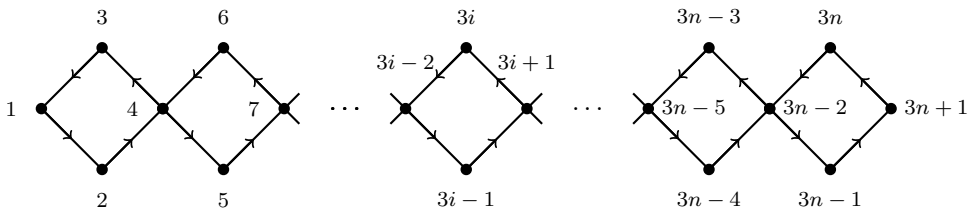
\begin{figure}[ht]
\begin{center}
\unitlength=.1mm
\begin{tikzpicture}[scale=0.4] 
%
\draw [line width=2pt, black] (-10,2) circle[radius=0.1cm];
\draw [line width=2pt, black] (-14,2) circle[radius=0.1cm];
\draw [line width=2pt, black] (-12,4) circle[radius=0.1cm];
\draw [line width=2pt, black] (-12,0) circle[radius=0.1cm];
\draw [thick, black] (-12,0)-- (-14,2) -- (-12,4) -- (-10,2) -- (-12,0);
\draw [thick, black,->] (-10,2)-- (-11,3);
\draw [thick, black,->] (-12,4)-- (-13,3);
\draw [thick, black,->] (-14,2)-- (-13,1);
\draw [thick, black,->] (-12,0)-- (-11,1);
\node (K) at (-15,2) {\tiny $1$};
\node (K) at (-11,2) {\tiny $4$};
\node (K) at (-12,5) {\tiny $3$};
\node (K) at (-12,-1) {\tiny $2$};
\draw [line width=2pt, black] (-6,2) circle[radius=0.1cm];
\draw [line width=2pt, black] (-10,2) circle[radius=0.1cm];
\draw [line width=2pt, black] (-8,4) circle[radius=0.1cm];
\draw [line width=2pt, black] (-8,0) circle[radius=0.1cm];
\draw [thick, black] (-8,0)-- (-10,2) -- (-8,4) -- (-6,2) -- (-8,0);
\draw [thick, black,->] (-6,2)-- (-7,3);
\draw [thick, black,->] (-8,4)-- (-9,3);
\draw [thick, black,->] (-10,2)-- (-9,1);
\draw [thick, black,->] (-8,0)-- (-7,1);
\draw [thick, black] (-6,2)-- (-5.5,2.5);
\draw [thick, black] (-6,2)-- (-5.5,1.5);
\node (K) at (-4,2) {$\ldots$};
\node (K) at (-7,2) {\tiny $7$};
\node (K) at (-8,5) {\tiny $6$};
\node (K) at (-8,-1) {\tiny $5$};
\draw [line width=2pt, black] (-2,2) circle[radius=0.1cm];
\draw [line width=2pt, black] (2,2) circle[radius=0.1cm];
\draw [line width=2pt, black] (0,4) circle[radius=0.1cm];
\draw [line width=2pt, black] (0,0) circle[radius=0.1cm];
\draw [thick, black] (0,0)-- (2,2) -- (0,4) -- (-2,2) -- (0,0);
\draw [thick, black,->] (0,0)-- (1,1);
\draw [thick, black,->] (2,2)-- (1,3);
\draw [thick, black,->] (0,4)-- (-1,3);
\draw [thick, black,->] (-2,2)-- (-1,1);
\draw [thick, black] (6,2)-- (5.5,2.5);
\draw [thick, black] (6,2)-- (5.5,1.5);
\node (K) at (-4,2) {$\ldots$};
\node (K) at (2,3.5) {\tiny $3i+1$};
\node (K) at (-2,3.5) {\tiny $3i-2$};
\node (K) at (0,5) {\tiny $3i$};
\node (K) at (0,-1) {\tiny $3i-1$};
\node (K) at (4,2) {$\ldots$};
\draw [line width=2pt, black] (6,2) circle[radius=0.1cm];
\draw [line width=2pt, black] (10,2) circle[radius=0.1cm];
\draw [line width=2pt, black] (8,4) circle[radius=0.1cm];
\draw [line width=2pt, black] (8,0) circle[radius=0.1cm];
\draw [thick, black] (8,0)-- (10,2) -- (8,4) -- (6,2) -- (8,0);
\draw [thick, black,->] (8,0)-- (9,1);
\draw [thick, black,->] (10,2)-- (9,3);
\draw [thick, black,->] (8,4)-- (7,3);
\draw [thick, black,->] (6,2)-- (7,1);
\draw [thick, black] (-2,2)-- (-2.5,2.5);
\draw [thick, black] (-2,2)-- (-2.5,1.5);
\draw [thick, black] (2,2)-- (2.5,2.5);
\draw [thick, black] (2,2)-- (2.5,1.5);
\node (K) at (-4,2) {$\ldots$};
\node (K) at (7.5,2) {\tiny $3n-5$};
\node (K) at (8,5) {\tiny $3n-3$};
\node (K) at (8,-1) {\tiny $3n-4$};
\draw [line width=2pt, black] (10,2) circle[radius=0.1cm];
\draw [line width=2pt, black] (14,2) circle[radius=0.1cm];
\draw [line width=2pt, black] (12,4) circle[radius=0.1cm];
\draw [line width=2pt, black] (12,0) circle[radius=0.1cm];
\draw [thick, black] (12,0)-- (14,2) -- (12,4) -- (10,2) -- (12,0);
\draw [thick, black,->] (12,0)-- (13,1);
\draw [thick, black,->] (14,2)-- (13,3);
\draw [thick, black,->] (12,4)-- (11,3);
\draw [thick, black,->] (10,2)-- (11,1);
\node (K) at (15.5,2) {\tiny $3n+1$};
\node (K) at (11.5,2) {\tiny $3n-2$};
\node (K) at (12,5) {\tiny $3n$};
\node (K) at (12,-1) {\tiny $3n-1$};
\end{tikzpicture}
\end{center}
\caption{Quiver $\Gamma$ with $3n+1$ vertices}  \label{quiver}
\end{figure}
Namely ${\bf B}$ is $(3n+1) \times (3n+1)$--matrix with enters determined by the quiver~$\Gamma$: 
$$
b_{ij} = \begin{cases} 
1, & \mbox{if there is an edge going from vertex $i$ to vertex $j$,}\\
-1, & \mbox{if there is an edge from vertex $j$ to vertex $i$,}\\
0, &  \mbox{if vertices $i$ and $j$ are not adjacent.}
\end{cases}
$$
In particular, if $n=2$ then matrix ${\bf B}$ is of the form
\begin{equation} 
{\bf B} = 
\begin{pmatrix}
0 & 1 & -1& 0& 0& 0& 0 \cr
-1 & 0 & 0 & 1 & 0& 0& 0 \cr
1 & 0 & 0 & -1 & 0& 0& 0\cr
0 & -1& 1 & 0 & 1 & -1 & 0 \cr
0& 0& 0& -1 & 0 & 0 & 1 \cr
0& 0& 0& 1 & 0 & 0 & -1 \cr
0& 0& 0& 0 & -1 & 1& 0
\end{pmatrix}. \label{matrixB}
\end{equation}

Let us denote by $\Phi : {\mathbb F}_{3n+1} \to {\mathbb F}_{3n+1}$, $n \geq 2$, the operator defined in~\cite[Formula~2-13]{HI15} as a composition of mutations. If $n=2$ then we get ${\bf x} = (x_1, x_2, x_3, x_4, x_5, x_6, x_7)$ and $\Phi$ is of the form    
	$$ 
	\Phi ({\bf x}) =  
\begin{pmatrix}  
	\Phi_1({\bf x}) \\ 
	\Phi_2({\bf x}) \\  
	\Phi_3({\bf x}) \\  
	\Phi_4({\bf x}) \\  
	\Phi_5({\bf x}) \\  
	\Phi_6({\bf x}) \\  
	\Phi_7({\bf x})  
	\end{pmatrix}^T
	= 
	\begin{pmatrix}
	x_1 \\ 
	x_5 \\  
	\frac{\displaystyle x_1 x_3 x_5 + x_3 x_4 x_5 + x_1 x_2 x_6}{\displaystyle x_2 x_4}  \\  
	\frac{\displaystyle x_1 x_3 x_4 x_5 + x_3 x_4^2 x_5 + x_1 x_3 x_5 x_7 + x_3 x_4 x_5 x_7 + x_1 x_2 x_6 x_7}{\displaystyle x_2 x_4 x_6} \\  
	\frac{\displaystyle x_1 x_3 x_5 + x_3 x_4 x_5 + x_1 x_2 x_6}{\displaystyle x_4 x_6} \\ 
	x_3 \\  
	x_7
	\end{pmatrix}^T.
	$$
We denote by $\Psi : {\mathbb F}_{3n+1} \to {\mathbb F}_{3n+1}$, $n \geq 2$, the operator inverse to $\Phi$. If $n=2$ then  
	$$
	\Psi({\bf x}) = 
	\begin{pmatrix} 
	\Psi_1({\bf x}) \\ 
	\Psi_2({\bf x}) \\ 
	\Psi_3({\bf x}) \\ 
	\Psi_4({\bf x}) \\  
	\Psi_5({\bf x}) \\  
	\Psi_6({\bf x}) \\  
	\Psi_7({\bf x})  
	\end{pmatrix}^T
	= 
	\begin{pmatrix}
	x_1 \\
	\frac{\displaystyle x_1 x_3 x_5 + x_1 x_2 x_6 + x_2 x_4 x_6}{\displaystyle x_3 x_4} \\ 
	x_6 \\ 
	\frac{\displaystyle x_1 x_2 x_4 x_6 + x_2 x_4^2 x_6 + x_1 x_3 x_5 x_7 + x_1 x_2 x_6 x_7 + x_2 x_4 x_6 x_7}{\displaystyle x_3 x_4 x_5} \\ 
	x_2 \\ 
	\frac{\displaystyle x_2 x_4 x_6 + x_3 x_5 x_7 + x_2 x_6 x_7}{\displaystyle x_4 x_5} \\ 
	x_7
	\end{pmatrix}^T.
	$$
Following~\cite[Formula~2-13]{HI15} we go from ${\bf x}$-variables to ${\bf y}$-variables. If $n=2$ then ${\bf y} = (y_1, y_2, y_3, y_4, y_5, y_6, y_7)$ and R-operator $\Phi$ will take a form $\varphi$, where  

\begin{equation}
\varphi ({\bf y}) =  
\begin{pmatrix}
\varphi_1({\bf y}) \\  
\varphi_2({\bf y}) \\  
\varphi_3({\bf y}) \\ 
\varphi_4({\bf y}) \\  
\varphi_5(\bf {y}) \\  
\varphi_6({\bf y}) \\  
\varphi_7({\bf y})  
\end{pmatrix}^T
= 
\begin{pmatrix}
y_1 (1 + y_2 + y_2 y_4) \\ 
\frac{\displaystyle y_2 y_4 y_5 y_6}{\displaystyle 1 + y_2 + y_6 + y_2 y_6 + y_2 y_4 y_6} \\  
\frac{\displaystyle 1 + y_2 + y_4 + y_2 y_6 + y_2 y_4 y_6}{\displaystyle y_2 y_4} \\ 
\frac{\displaystyle y_4}{\displaystyle (1 + y_2 + y_2 y_4) (1 + y_6 + y_4 y_6)} \\  
\frac{\displaystyle 1 + y_2 + y_6 + y_2 y_6 + y_2 y_4 y_6}{\displaystyle y_4 y_6} \\  
\frac{\displaystyle y_2 y_3 y_4 y_6}{\displaystyle 1 + y_2 + y_6 + y_2 y_6 + y_2 y_4 y_6} \\  
(1 + y_6 + y_4 y_6) y_7
\end{pmatrix}^T, \label{ytrans}
\end{equation}
as well as $\Psi$ will takes a form $\psi$, where 
\begin{equation}
\psi({\bf y}) = 
\begin{pmatrix}
\psi_1({\bf y}) \\  
\psi_2({\bf y}) \\  
\psi_3({\bf y}) \\  
\psi_4({\bf y}) \\  
\psi_5({\bf y}) \\  
\psi_6({\bf y}) \\  
\psi_7({\bf y})  
\end{pmatrix}^T
= 
\begin{pmatrix}
\frac{\displaystyle y_1 y_3 y_4}{\displaystyle 1 + y_4 +  y_3 y_4} \\ 
\frac{\displaystyle  y_5}{\displaystyle 1 + y_4 + y_3 y_4 + y_4 y_5 + y_3 y_4 y_5} \\ 
(1 + y_4 + y_3 y_4 + y_4 y_5+ y_3 y_4 y_5) y_6 \\ 
\frac{\displaystyle (1 + y_4 + y_3 y_4)(1 + y_4 + y_4 y_5)}{\displaystyle y_3 y_4 y_5} \\  
y_2 (1 + y_4 + y_3 y_4 + y_4 y_5 + y_3 y_4 y_5) \\ 
\frac{\displaystyle  y_3 }{\displaystyle 1 + y_4+ y_3 y_4+ y_4 y_5 + y_3 y_4 y_5} \\  
\frac{\displaystyle y_4 y_5 y_7}{\displaystyle 1 + y_4 + y_4 y_5}. 
\end{pmatrix}^T. \label{y-1trans}
\end{equation}
The following property easily follows from the above formulae.
\begin{lemma}
By setting $y_1 = y_4 = y_7 = -1$ in formulae (\ref{ytrans}) and  (\ref{y-1trans}) we get 
$$
\varphi_1  = \varphi_4 = \varphi_7   = -1 \qquad \mbox{and} \qquad \psi_1  = \psi_4  = \psi_7   = -1.
$$  
\end{lemma}

\section{Virtual braid groups}

For a vector  ${\bf z} = (z_1, z_2, z_3, z_4) = (y_2, y_3, y_5, y_6)$ of length four we define two operators    	
 \begin{equation}
S  
\begin{pmatrix} 
z_1 \\ 
z_2 \\ 
z_3 \\ 
z_4 \\
\end{pmatrix}^T = 
\begin{pmatrix} 
- \frac{\displaystyle z_1 z_3 z_4}{\displaystyle 1 + z_1 + z_4} \\ 
-\frac{\displaystyle 1 + z_1 + z_4}{\displaystyle z_1} \\ 
-\frac{\displaystyle 1 + z_1 + z_4}{\displaystyle z_4} \\ 
-\frac{\displaystyle z_1 z_2 z_4}{\displaystyle 1 + z_1 + z_4} 
\end{pmatrix}^T,  \quad 
S^{-1} 
\begin{pmatrix} 
z_1 \\ 
z_2 \\ 
z_3 \\ 
z_4 \\
\end{pmatrix}^T = 
\begin{pmatrix}
- \frac{\displaystyle z_3}{\displaystyle z_2 + z_3 + z_2 z_3} \\\ 
- (z_2 + z_3 + z_2 z_3) z_4 \\ 
- z_1 (z_2 + z_3 + z_2 z_3)  \\ 
-\frac{\displaystyle z_2}{\displaystyle z_2 + z_3 + z_2 z_3}  
\end{pmatrix}^T \label{S-eqn} 
\end{equation}
 and an involution   
\begin{equation}
T (z_1, z_2, z_3, z_4) = (z_3, z_4, z_1, z_2). \label{T-eqn}
\end{equation}	
Now for $n\geq 2$ we define operators $S_i^{\pm 1}$ and $T_i$, $i=1, \ldots, n-1$, which act on  vector ${\bf z} = (z_1, z_2, \ldots, z_{2n})$ of length $2n$ by the following rule. Operators $S_i^{\pm 1}$ and $T_i$ act on 4-tuple  $(z_{2i-1}, z_{2i}, z_{2i+1}, z_{2i+2})$ in the same way as operators $S^{\pm 1}$ and $T$ act on 4-tuple $(z_1, z_2, z_3, z_4)$, and do not change other components of ${\bf z}$: 
$$
S_i^{\pm 1} = I^{2i-2} \otimes S^{\pm 1} \otimes I^{2n-2i-2}, \qquad T_i = I^{2i-2} \otimes T \otimes I^{2n-2i-2}.
$$ 

For $n \geq 2$ we denote by $\Theta_n$ the group generated by $S_i$, $T_i$, $i=1, \ldots, n-1$, with composition as a  group operation. Define a map   $ F : VB_n \to \Theta_n$ by setting  
\begin{equation}
F (\sigma_i) = S_i, \qquad F(\rho_i) = T_i, \qquad i=1, \ldots, n-1.  \label{map}
\end{equation}

\begin{lemma} \label{l1}
Let $w$ be a word in $VB_n$. Then for a vector of algebraically independent variables ${\bf z} = (z_1, z_2, \ldots, z_{2n})$ in the image of $F(w) ({\bf z})$ no coordinate turns into zero or infinity.
\end{lemma}

\begin{proof}
Consider $2n$-tuple ${\bf z}' = (-1, -1, \ldots, -1)$. It is easy to see from (\ref{S-eqn}) and (\ref{T-eqn}) that $S_i^{\pm 1} ({\bf z}') = {\bf z}'$ and $T_i ({\bf z}') = {\bf z}'$ for each $i$. Hence $F(w) ({\bf z}') = {\bf z}' = (-1, -1, \ldots, -1)$. Hence, in the image of $F(w)({\bf z})$ no coordinate can turn into zero or infinity, because for $z_i = -1$, $i=1, \ldots, 2n$, all coordinates of the image will be equal to $-1$.
\end{proof}	
	 
\begin{theorem} \label{th1}
Map $F : VB_n \to \Theta_n$, $n \geq 2$, defined by (\ref{map}) is a homomorphism.  
\end{theorem}

\begin{proof}
Let us check that operators $S_i$ and $T_i$, $i=1, \ldots, n-1$, act on ${\bf z}$ in such a way that the following identities hold.  
\begin{itemize}
\item[(1)] $S_i S_{i+1} S_i = S_{i+1} S_i S_{i+1}$, where $i = 1, 2, \ldots, n-2$.
\item[(2)] $S_i S_j  = S_j S_i$, where $| i - j | \geq 2$.
\item[(3)] $T_i T_{i+1} T_i = T_{i+1} T_i T_{i+1}$, where $i = 1, 2, \ldots, n-2$.
\item[(4)] $T_i T_j = T_j T_i$, where $| i - j | \geq 2$.
\item[(5)] $T_i^2 = 1$, where $i = 1, 2, \ldots, n-1$.
\item[(6)] $T_i T_{i+1} S_i = S_{i+1} T_i T_{i+1}$, where $i = 1, 2, \ldots, n-2$.
\end{itemize} 
Obviously, it is enough to consider the case $i=1$. Identities (1) and (2) are particular cases of~\cite[Theorem~2.3]{HI15}. Nevertheless, we present a strightforward proof of (1) for a reader convenience.  Let ${\bf z} = (z_1, z_2, z_3, z_4, z_5, z_6)$. Consider the left-side part of (1)  
\small{$$
\begin{gathered}
S_1 S_2 S_1 ({\bf z}) = S_1 S_2 S_1(z_1, z_2, z_3, z_4, z_5, z_6) \\
=S_1 S_2 \left( -\frac{z_1 z_3 z_4}{1 + z_1 + z_4}, -\frac{1 + z_1 + z_4}{z_1}, -\frac{1 + z_1 + z_4}{z_4}, -\frac{z_1 z_2 z_4}{1 + z_1 + z_4}, z_5, z_6 \right) \\
= S_1 \left( -\frac{z_1 z_3 z_4}{1 + z_1 + z_4}, -\frac{1 + z_1 + z_4}{z_1}, -\frac{(1 + z_1 + z_4) z_5 z_6}{1 + z_1 - z_4 z_6}, -\frac{1 + z_1 - z_4 z_6}{1 + z_1 + z_4}, \frac{1 + z_1 - z_4 z_6}{z_4 z_6}, \frac{z_1 z_2 z_4 z_6}{1 + z_1 - z_4 z_6} \right) \\ 
 = \left( \frac{z_1 z_3 z_5 z_6}{1 - z_1 z_3 + z_6}, \frac{1 - z_1 z_3 + z_6}{z_1 z_3}, \frac{z_4 (1 - z_1 z_3 + z_6)}{1 + z_1 - z_4 z_6}, \frac{z_3 (1 + z_1 - z_4 z_6)}{1 - z_1 z_3 + z_6},  \frac{1 + z_1 - z_4 z_6}{z_4 z_6},  \frac{z_1 z_2 z_4 z_6}{1 + z_1 - z_4 z_6} \right).
\end{gathered}
$$
}
The right-side part of (1) is equal 
\small{$$
\begin{gathered}
S_2 S_1 S_2 ({\bf z})  =  S_2 S_1 S_2 ( z_1, z_2, z_3, z_4, z_5, z_6) \\
=S_2 S_1 \left( z_1, z_2, -\frac{z_3 z_5 z_6}{1 + z_3 + z_6}, -\frac{1 + z_3 + z_6}{z_3}, -\frac{1 + z_3 + z_6}{z_6}, -\frac{z_3 z_4 z_6}{1 + z_3 + z_6} \right) \\ 
 = S_2 \left( \frac{z_1 z_2 z_5 z_6}{1 - z_1 z_2 + z_6}, \frac{1 - z_1 z_3 + z_6}{z_1 z_3}, -\frac{1 - z_1 z_3 + z_6}{1 + z_3 + z_6}, \frac{z_1 z_2 (1 + z_3 + z_6)}{-1 + z_1 z_3 - z_6}, -\frac{1 + z_3 + z_6}{z_6},  -\frac{z_3 z_4 z_6}{1 + z_3 + z_6} \right) \\
= \left( \frac{z_1 z_3 z_5 z_6}{1 - z_1 z_3 + z_6}, \frac{1 - z_1 z_3 + z_6}{z_1 z_3}, \frac{z_4 (1 - z_1 z_3 + z_6)}{1 + z_1 - z_4 z_6}, \frac{z_3 (1 + z_1 - z_4 z_6)}{1 - z_1 z_3 + z_6},  \frac{1 + z_1 - z_4 z_6}{z_4 z_6}, \frac{z_1 z_2 z_4 z_6}{1 + z_1 - z_4 z_6} \right).
\end{gathered}	
$$
}
Thus, the identity (1) holds. 

Let us demonstrate that the identity (6) holds, Indeed from the one hand, 
$$
\begin{gathered} 
T_1 T_2 S_1 ({\bf z}) = 
T_1 T_2 S_1 (z_1, z_2, z_3, z_4, z_5, z_6) = T_1 T_2 (S(z_1), S(z_2), S(z_3), S(z_4), z_5, z_6) \\
= T_1 (S(z_1), S(z_2), z_5, z_7, S(z_3), S(z_4)) =  (z_5, z_6, S(z_1), S(z_2), S(z_3), S(z_4) ), 
\end{gathered}
$$
and from the other hand, 
$$
\begin{gathered} 
S_2 T_1 T_2 ({\bf z}) = 
S_2 T_1 T_2 (z_1, z_2, z_3, z_4, z_5, z_6) = S_2 T_1 (z_1, z_2, z_5, z_6, z_3, z_4) \\
= S_2 (z_5, z_6, z_1, z_2, z_3, z_4) =  (z_5, z_6, S(z_1), S(z_2), S(z_3), S(z_4) ).
\end{gathered}
$$
Remaining identities  (2), (3), (4) and (5) hold obviously. 
\end{proof}

Theorem~\ref{th1} allows to distinguish elements of the virtual braid group $VB_n$ by computing their images which are vectors of lengths  $2n$. 
\begin{example}
{\rm
Let $w_1 = \sigma_1 \rho_1 \sigma_1 \in VB_2$.  By formulae (\ref{S-eqn}) and (\ref{T-eqn}) the operator $F(w_1)$ acts on $ (1, 2, 2, 1)$ in the following way: 
$$
F(w_1) (1, 2, 2, 1)= \left(-\frac{6}{5}, -\frac{5}{3}, -\frac{5}{3}, -\frac{6}{5} \right) \neq (1, 2, 2, 1).
$$ 
Therefore, the homomorphism $F$ distinguishes $w_1$ from a trivial braid. 
}
\end{example} 	

\begin{example}
{\rm 
It is known~\cite{Ma05} that a generalized Burau representation does nor distinguish a braid $w_2=(\sigma_1^2 \rho_1 \sigma_1^{-1} \rho_1 \sigma_1^{-1} \rho_1)^2 \in VB_2$ from a trivial braid.  By acting $F(w_2)$ on the vector $ (1, 2, 2, 1)$ we get  
$$
F(w_2) (1, 2, 2, 1)= \left( -\frac{44}{19}, -\frac{19}{22}, -\frac{19}{22}, -\frac{44}{19} \right) \neq (1, 2, 2, 1).
$$ 
Therefore, the homomorphism $F$ distinguishes $w_2$ from a trivial braid. 
}
\end{example} 

\begin{example}
{\rm 
Consider 
$$
w_3=\sigma_1 \rho_2 \sigma_1 \sigma_2^{-1} \sigma_1 \sigma_2 \sigma_1^{-1} \rho_1 \sigma_2 \rho_1 \sigma_1 \rho_2 \sigma_1^{-1} \rho_2 \sigma_2^{-1} \sigma_1^{-1} \sigma_2 \sigma_1^{-1}   \rho_2 \sigma_1^{-1}  \in VB_3.
$$  
It is known that a representation from~\cite{BVW12} does not distinguish $w_3$ from a trivial braid. By acting $F(w_3)$ on $ (1, 2, 2, 1, 1, 2)$ we get  
\small{
$$
\begin{gathered}
F(w_3) (1, 2, 2, 1, 1, 2) = \left( \frac{2488285076682521504}{1290542656863845663}, \frac{1290542656863845663}{1244142538341260752}, \right. \\ \left. \frac{1290542656863845663}{563568067426145589}, \frac{1127136134852291178}{1290542656863845663}, \frac{574648281}{1268603408}, \frac{2537206816}{574648281} \right) \neq (1, 2, 2, 1, 1, 2).
\end{gathered}		
$$
} 
Therefore, the homomorphism $F$ distinguishes $w_3$ from a trivial braid. 
}
\end{example} 
	
\section{Forbidden relations}	

In this section we demonstrate that the forbidden relations do not hold in the group  $\Theta_n$. 	
	
\begin{lemma} \label{lemma2}
Let ${\bf z} = (z_1, z_2,  \ldots, z_{2n-1}, z_{2n})$ and $S_{i}, S_{i+1}, T_{i}, T_{i+1} \in \Theta_n$.   
\begin{itemize}
\item[(a)] The forbidden relation 
 \begin{equation}
T_{i} S_{i+1} S_{i} ({\bf z}) = S_{i+1} S_{i} T_{i+1} ({\bf z}) \label{f1}
\end{equation}
does not hold if and only if the vector  ${\bf z}$ is such that $z_j \neq -1$ for $j=2i-1$, $2i+2$ and $2i+4$. 
\item[(b)] The forbidden relation 
\begin{equation}
T_{i+1} S_{i} S_{i+1}  ({\bf z})  = S_{i} S_{i+1} T_{i+1}  ({\bf z})  \label{f2}
\end{equation}
does not hold if and only if the vector  ${\bf z}$ is such that $z_j \neq -1$ for $j=2i-1$, $2i+1$ and $2i+4$. 
\end{itemize} 
\end{lemma}

\begin{proof} (a) Without loss of generality, we can assume  $i=1$.  The left-hand side of (\ref{f1}) is    
\small{$$
\begin{gathered}
T_1 S_2 S_1 ({\bf z}) = T_1 S_2 S_1 (z_1, z_2, z_3, z_4, z_5, z_6) \\
 = T_1 \left( 
 - \frac{z_1 z_3 z_4}{1 + z_1 + z_4}, - \frac{1 + z_1 + z_4}{z_1}, - \frac{(1 + z_1 + z_4) z_5 z_6}{1 + z_1 - z_4 z_6}, - \frac{1 + z_1 - z_4 z_6}{1 + z_1 + z_4},  \frac{1 + z_1 - z_4 z_6}{z_4 z_6}, \frac{z_1 z_2 z_3 z_6}{1+z_1-z_4 z_6}
\right) \\
 = \left(
- \frac{(1 + z_1 + z_4) z_5 z_6}{1 + z_1 - z_4 z_6}, - \frac{1 + z_1 - z_4 z_6}{1 + z_1 + z_4}, - \frac{z_1 z_3 z_4}{1 + z_1 + z_4}, - \frac{1 + z_1 + z_4}{z_1}, \frac{1 + z_1 - z_4 z_6}{z_4 z_6}, \frac{z_1 z_2 z_4 z_6}{1 + z_1 -z_4 z_6}
\right)
\end{gathered}
$$}
and the right-hand side is equal
\small{$$
\begin{gathered}
S_2 S_1 T_2 ({\bf z}) =  S_2 S_1 T_2 (z_1, z_2, z_3, z_4, z_5, z_6)  = S_2 S_1 (z_1, z_2, z_5, z_6, z_3, z_4) \\
= \left(
- \frac{z_1 z_5 z_6}{1 + z_1  + z_6}, - \frac{1 + z_1 +z_6}{z_1}, - \frac{(1 + z_1 + z_6) z_3 z_4}{1 + z_1  - z_4 z_6}, - \frac{1 + z_1 - z_6 z_4}{1 + z_1 + z_6}, \frac{1 + z_1 - z_6 z_4}{z_6 z_4}, \frac{z_1 z_2 z_6 z_4}{1+z_1- z_6 z_4}
\right)
\end{gathered}.
$$}
Here we used formulae for $S_2 S_1 ({\bf z})$ from theorem~\ref{th1}. The fifth and sixth coordinates are equal. Comparison of the third and fourth coordinates leads to the equation 
\begin{equation}
z_1 (1+z_1 - z_4 z_6) = (1 + z_1 + z_4) (1+z_1+z_6), \label{eqn1}
\end{equation}
which is equivalent to  
\begin{equation}
(z_1+1) (z_4+1) (z_6+1) = 0. \label{eqn2}
\end{equation}
Therefore, to obtain the relation (a) the necessary condition is  that at least one of number  $z_1$, $z_4$ or $z_6$ is equal to $-1$. But if at least one of numbers  $z_1$, $z_4$ or $z_6$ is equal to  $-1$, then the left and right parts of (a) coincide. Indeed, if $z_1 = -1$ then
$$
T_1 S_2 S_1 (-1, z_2, z_3, z_4, z_5. z_6) = S_2 S_1 T_2 (-1, z_2, z_3, z_4, z_5, z_6) =  (z_5, z_6, z_3, z_4, -1, z_2).
$$
Analogously, if $z_4 = -1$ then
$$
\begin{gathered} 
T_1 S_2 S_1 (z_1, z_2, z_3, -1, z_5, z_6) = S_2 S_1 T_2 (z_1, z_2, z_3, -1, z_5, z_6) \cr
 = \left(-\frac{z_1 z_5 z_6}{1 + z_1 + z_6}, - \frac{1 + z_1 + z_6}{z_1}, z_3, -1, - \frac{1 + z_2 + z_6}{z_6}, - \frac{z_1 z_2 z_6}{1+z_1 + z_6} \right),  
\end{gathered}
$$
and if $z_6 = -1$ then  
$$
\begin{gathered}
T_1 S_2 S_1 (z_1, z_2, z_3, z_4, z_5, -1) =  S_2 S_1 T_1 (z_1, z_2, z_3, z_4, z_5, -1) \cr 
= \left(z_5, -1, - \frac{z_1 z_3 z_4}{1 + z_1 + z_4},  - \frac{1 + z_2 + z_4}{z_1}, - \frac{1 + z_1 + z_4}{z_4}, - \frac{z_1 z_2 z_4}{1 + z_1 + z_4} \right).
\end{gathered}
$$
Therefore, the above necessary condition is also sufficient. 

(b) The left-hand part of the relation (\ref{f2}) is equal to  
\small{$$
\begin{gathered}
T_2 S_1 S_2 ({\bf z}) = T_2 S_1 S_2 (z_1, z_2, z_3, z_4, z_5, z_6) \\
 = T_2 \left( 
 \frac{z_1 z_3 z_5 z_6}{1 - z_1z_3 + z_6},  \frac{1 -z_1z_3 + z_6}{z_1 z_3}, -\frac{1 - z_1 z_3 + z_6}{1 + z_3 + z_6}, - \frac{z_1 z_2 (1+z_3+z_6)}{1 - z_1 z_3 + z_6},  - \frac{1 + z_3 + z_6}{z_6}, - \frac{z_3 z_4 z_6}{1+z_3 +z_6}
\right) \\
 = \left(
 \frac{z_1 z_3 z_5 z_6}{1 - z_1 z_3 + z_6},  \frac{1 -z_1 z_3 + z_6}{z_1 z_3},  - \frac{1 + z_3 + z_6}{z_6}, - \frac{z_3 z_4 z_6}{1+z_3  +z_6}, -\frac{1 - z_1 z_3 + z_6}{1 + z_3 + z_6}, - \frac{z_1 z_2 (1+z_3+z_6)}{1 - z_1 z_3 + z_6} 
\right)
\end{gathered}
$$}
and the right-hand part is equal to  
\small{$$
\begin{gathered}
S_1 S_2 T_1 ({\bf z}) =  S_1 S_2 T_1 (z_1, z_2, z_3, z_4, z_5, z_6)  = S_2 S_1 (z_3, z_4,  z_1, z_2, z_5, z_6) \\
= \left( 
 \frac{z_3 z_1 z_5 z_6}{1 - z_3 z_1 + z_6},  \frac{1 -z_3 z_1 + z_6}{z_3 z_1}, - \frac{1 - z_1 z_3 + z_6}{1 + z_1 + z_6}, - \frac{z_3 z_4 (1+z_1+z_6)}{1 - z_3 z_1 + z_6},  - \frac{1 + z_1 + z_6}{z_6}, - \frac{z_1 z_2 z_6}{1+z_1 +z_6}
\right).
\end{gathered}
$$}
Here we used formulae for  $S_1 S_2 ({\bf z})$ from theorem~\ref{th1}. By comparing the third and fourth coordinates we get the equation 
\begin{equation}
(1 + z_3 + z_6)(1 + z_1 + z_6)  =  z_6(1 - z_1 z_3 + z_6), \label{eqn4}
\end{equation}
which is equivalent to  
\begin{equation}
(z_1+1) (z_3+1) (z_6+1) = 0. \label{eqn5}
\end{equation}
Therefore, to obtain (b) the necessary condition is that at least one of  $z_1$, $z_3$ or $z_6$ is equal to  $-1$. But if at least on of this numbers is equl to  $-1$ then left and right parts of  (\ref{f2}) coincide. Indeed, if $z_1 = -1$ then
$$
\begin{gathered} 
T_2 S_1 S_2 (-1, z_2, z_3, z_4, z_5, z_6) = S_2 S_1 T_2 (-1, z_2, z_3, z_4, z_5, z_6) \cr 
 = \left( - \frac{z_3 z_5 z_6}{1 + z_3 + z_6}, - \frac{1 + z_3 + z_6}{z_3}, - \frac{1 + z_3 + z_6}{z_6}, - \frac{z_3 z_4 z_6}{1 + z_3 + z_6}, -1, z_2 \right).
\end{gathered}
$$
Analogously, if $z_3 = -1$ then  
$$
\begin{gathered}
T_2 S_1 S_2 (z_1, z_2, -1, z_4, z_5, z_6) = S_2 S_1 T_2 (z_1, z_2, -1, z_4, z_5, z_6) \cr 
= \left( - \frac{z_1 z_5 z_6}{1 + z_3 + z_6}, - \frac{1 + z_1 + z_6}{z_1}, -1, z_4, - \frac{1 + z_1 + z_6}{z_6}, - \frac{z_1 z_2 z_6}{1 + z_1 + z_6} \right) 
\end{gathered}
$$
and if $z_6 = -1$ then 
$$
T_2 S_1 S_2 (z_1, z_2, z_3, z_4, z_5, -1) = S_2 S_1 T_2 (z_1, z_2, z_3, z_4, z_5, -1) = (z_5, -1, z_3, z_4, z_1, z_2). 
$$
Therefore, the above necessary condition is also sufficient. 
\end{proof}

The obvious consequence of this lemma is the following theorem, which concludes the section.
\begin{theorem}
Let $S_{i}, S_{i+1}, T_{i}, T_{i+1} \in \Theta_n$. 
\begin{itemize}
	\item[(a)] Operators $T_{i} S_{i+1} S_{i}$ and $S_{i+1} S_{i} T_{i+1}$ are different.
	\item[(b)] Operators $T_{i+1} S_{i} S_{i+1}$ and $S_{i} S_{i+1} T_{i+1}$ are different.
\end{itemize}	
So the forbidden relations do not hold in $\Theta_n$. 
\end{theorem}

\section{Flat braid groups}

Let us consider vector ${\bf z}$ of the form $(z_1, \frac{1}{z_1}, z_3, \frac{1}{z_3})$. Notice that
$$
S({\bf z}) = ( \zeta_1, \frac{1}{\zeta_1}, \zeta_3, \frac{1}{\zeta_3}), 
$$
where 
$$
\zeta_1 = - \frac{z_1 z_3}{1 + z_3 + z_1 z_3}, \qquad \zeta_3 = - (1 + z_3 + z_1 z_3 ).
$$
Also notice that $S^2({\bf z})={\bf z}$.
These obserations inspire to obtain the representation for flat braids.

Consider a vector of algebraically independent variables ${\bf t} = (t_1, t_2, \ldots, t_n)$. Let's define the operators $R_i$, $i = 1, \ldots, n-1$, according to the rule:
$$
R_i : \begin{cases}
t_i \to & - \frac{\displaystyle t_{i} t_{i+1}}{\displaystyle 1 + t_{i+1} + t_i t_{i+1}}, \\
t_{i+1} \to & - (1 + t_{i+1} + t_i t_{i+1}).
\end{cases}
$$ 

Let $F_{FB}$ be a map that match operators $R_i$ with generators $\sigma_i$, $i=1, \ldots, n-1,$ of the flat braid group $T_n$:
$$
F_{FB} (\sigma_i) = R_i.
$$

For $n\geq 2$, denote by $\Omega_n$ the group generated by operators $R_i$, $i=1, \ldots, n-1$, with composition as a group operation.

\begin{lemma} \label{l2}
	Let $w$ be a word in $FB_n$. Then for a vector of algebraically independent variables ${\bf t} = (t_1, t_2, \ldots, t_{n})$ in the image of $F_{FB}(w) ({\bf t})$ no coordinate turns into zero or infinity.
\end{lemma}

\begin{proof}
	Consider $n$-tuple ${\bf t}' = (-1, -1, \ldots, -1)$. It is easy to see that $R_i^{\pm 1} ({\bf t}') = {\bf t}'$ for each $i$. Hence $F_{FB}(w) ({\bf t}') = {\bf t}'$. Hence, in the image of $F_{FB}(w) ({\bf t})$ no coordinate can turn into zero or infinity, because for  $t_i = -1$, $i=1, \ldots, n$, all coordinates of the image will be equal to $-1$.
\end{proof}	

\begin{theorem} \label{th3}
	Correspondence  $F_{FB} : FB_n \to \Omega_n$ is a homomorphism for any $n \geq 2$.  
\end{theorem}

\begin{proof}
Let us check that for the operators $R_i$, $i=1, \ldots, n-1$, act on ${\bf t}$ in such a way that the following identities hold.
	\begin{itemize}
		\item[(1)] $R_i^2 = 1$, where $i = 1, 2, \ldots, n-2$.
		\item[(2)] $R_i R_{i+1} R_i = R_{i+1} R_i R_{i+1}$, where $i = 1, 2, \ldots, n-2$.
		\item[(3)] $R_i R_j  = R_j R_i$, where $| i - j | \geq 2$.
	\end{itemize} 
	
	We present a proof for the case of $i=1$, which also works for an arbitrary $i= 1, \ldots, n-1$. Consider ${\bf t} = (t_1, t_2, t_3)$. The relation (1) is easily verified. Indeed:
	$$
	R_1^2 ({\bf t}) = R_1^2(t_1, t_2, t_3) = R_1 (-\frac{t_{1} t_{2}}{1 + t_{1} + t_1 t_{2}}, -(1 + t_{2} + t_1 t_{2}), t_3) = ((t_1, t_2, t_3))
	$$
	
	Let us now prove identity (2). Its left-hand side is
	\small{$$
		\begin{gathered}
		R_1 R_2 R_1 ({\bf t}) = R_1 R_2 R_1(t_1, t_2, t_3)
		=R_1 R_2 \left( -\frac{t_{1} t_{2}}{1 + t_{1} + t_1 t_{2}}, -(1 + t_{2} + t_1 t_{2}), t_3 \right) \\
		= R_1 \left( -\frac{t_{1} t_{2}}{1 + t_{1} + t_1 t_{2}}, -\frac{1 + t_2 t_3 + t_1 t_2 t_3}{-1 + t_2 t_3 + t_1 t_2 t_3}, -1 + t_2 t_3 + t_1 t_2 t_3  \right) \\ 
		= \left(\frac{t_{1} t_{2} t_3}{1 + t_{3} - t_1 t_{2} t_3}, \frac{1 + t_3 - t_1 t_2 t_3}{-1 + t_2 t_3 + t_1 t_2 t_3}, -1 + t_2 t_3 + t_1 t_2 t_3  \right).
		\end{gathered}
		$$
	}
The rifght-hand side is
	\small{$$
	\begin{gathered}
	R_2 R_1 R_2 ({\bf t}) = R_2 R_1 R_2(t_1, t_2, t_3)
	=R_2 R_1 \left( -\frac{t_1, t_{2} t_{3}}{1 + t_{2} + t_2 t_{3}}, -(1 + t_{3} + t_2 t_{3}) \right) \\
	= R_2 \left(\frac{t_{1} t_{2} t_3}{1 + t_{3} - t_1 t_{2} t_3}, -\frac{1 + t_3 - t_1 t_2 t_3}{1 + t_3 + t_2 t_3}, -(1 + t_{3} + t_2 t_{3}) \right) \\ 
	= \left(\frac{t_{1} t_{2} t_3}{1 + t_{3} - t_1 t_{2} t_3}, \frac{1 + t_3 - t_1 t_2 t_3}{-1 + t_2 t_3 + t_1 t_2 t_3}, -1 + t_2 t_3 + t_1 t_2 t_3  \right).
	\end{gathered}
	$$
}
Thus, identity (2) holds.
The fulfillment of identity (3) is obvious.
\end{proof}
	
\begin{example}
{\rm
Consider $w_4=  (\sigma_1 \sigma_2)^2  \in FB_3$. The operator  $F_{FB}(w_4)$ acts on $ (1, 2, 2)$ in the following way;  
$$
F_{FB}(w_4) (1, 2, 2)= \left(-\frac{2}{5}, -\frac{10}{7}, 7 \right) \neq (1, 2, 2).
$$ 
Therefore, the homomorphism $F_{FB}$ distinguishes $w_4$ from a trivial braid. 
	}
\end{example}	
	


\section{Flat virtual braid groups}

Consider a vector of algebraically independent variables ${\bf t} = (t_1, t_2, \ldots, t_n)$.  In addition to the operators $R_i$, $i = 1, \ldots, n-1$ introduced in the previous section, we define the operators $V_i$, $i = 1, \ldots, n-1$, according to the rule:
$$
V_i : \begin{cases}
t_i \to &  t_{i+1}, \\
t_{i+1} \to & t_i.
\end{cases}
$$ 
Let $F_{FVB}$ be a map that match operators $R_i$ and $V_i$ with generators $\sigma_i$ and $\rho_i$, $i=1, \ldots, n-1$, of the virtual flat braid group $T_n$:
$$
F_{FVB} (\sigma_i) = R_i, \qquad F_{FVB}(\rho_i) = V_i.
$$

For $n \geq 2$, denote by $\Delta_n$ the group generated by operators $R_i,V_i$, $i=1, \ldots, n-1$, with composition as a group operation.

\begin{lemma} \label{l3}
	Let $w$ be a word in $FVB_n$. Then for a vector of algebraically independent variables ${\bf t} = (t_1, t_2, \ldots, t_{n})$ in the image of $F_{FVB}(w) ({\bf t})$ no coordinate turns into zero or infinity.
\end{lemma}

\begin{proof}
		Consider $n$-tuple ${\bf t}' = (-1, -1, \ldots, -1)$. It is easy to see that $R_i^{\pm 1} ({\bf t}') = {\bf t}'$ and  $V_i^{\pm 1} ({\bf t}') = {\bf t}'$ for each $i$. Hence $F_{FVB}(w) ({\bf t}') = {\bf t}'$. Hence, in the image of $F_{FVB}(w) ({\bf t})$ no coordinate can turn into zero or infinity, because for  $t_i = -1$, $i=1, \ldots, n$, all coordinates of the image will be equal to $-1$.
\end{proof}

\begin{theorem} \label{th4}
	Correspondence  $F_{FVB} : FVB_n \to \Delta_n$ is a homomorphism for any $n \geq 2$.  
\end{theorem}

\begin{proof}
	Let us check that the operators $R_i$ and $V_i$, $i=1, \ldots, n-1$, act on ${\bf t}$ in such a way that the following identities hold.
	\begin{itemize}
		\item[(1)] $R_i^2 = 1$, where $i = 1, 2, \ldots, n-2$.
		\item[(2)] $R_i R_{i+1} R_i = R_{i+1} R_i R_{i+1}$, where $i = 1, 2, \ldots, n-2$.
		\item[(3)] $R_i R_j  = R_j R_i$, where $| i - j | \geq 2$.
		\item[(4)] $V_i V_{i+1} V_i = V_{i+1} V_i V_{i+1}$, where $i = 1, 2, \ldots, n-2$.
		\item[(5)] $V_i V_j = V_j V_i$, where $| i - j | \geq 2$.
		\item[(6)] $V_i^2 = 1$, where $i = 1, 2, \ldots, n-1$.
		\item[(7)] $V_i V_{i+1} R_i = R_{i+1} V_i V_{i+1}$, where $i = 1, 2, \ldots, n-2$.
	\end{itemize} 
Identities (1), (2) and (3) are proved in theorem~\ref{th3}. The fulfillment of identities (5) and (6) is obvious. It remains to prove the relations (4) and (7). We present a proof for the case of $i=1$, which also works for an arbitrary $i= 1, \ldots, n-1$. Consider ${\bf t} = (t_1, t_2, t_3)$.  Let us now prove the identity (4). Its left-hand side is
$$
V_1 V_2 V_1 ({\bf t}) = V_1 V_2 V_1 (t_1, t_2, t_3) = V_1 V_2 (t_2, t_1, t_3) = V_1 (t_2, t_3, t_1) = (t_3, t_2, t_1).
$$

The right-hand side is
$$
V_2 V_1 V_2 ({\bf t}) = V_2 V_1 V_2 (t_1, t_2, t_3) = V_2 V_1 (t_1, t_3, t_2) = V_2 (t_3, t_1, t_2) = (t_3, t_2, t_1).
$$
So the identity (4) holds. Let us now prove the identity (7). Its left-hand side is
	\small{$$
	\begin{gathered}
	V_1 V_2 R_1 ({\bf t}) = V_1 V_2 R_1(t_1, t_2, t_3)
	=V_1 V_2 \left( -\frac{t_{1} t_{2}}{1 + t_{1} + t_1 t_{2}}, -(1 + t_{2} + t_1 t_{2}), t_3 \right) \\
	= V_1 \left( -\frac{t_{1} t_{2}}{1 + t_{1} + t_1 t_{2}}, t_3, -(1 + t_{2} + t_1 t_{2}) \right)  
	=  \left( t_3, -\frac{t_{1} t_{2}}{1 + t_{1} + t_1 t_{2}}, -(1 + t_{2} + t_1 t_{2}) \right).
	\end{gathered}
	$$
}

The right-hand side is
	\small{$$
	\begin{gathered}
	R_2 V_1 V_2 ({\bf t}) = R_2 V_1 V_2(t_1, t_2, t_3)
	=R_2 R_1 \left(t_1, t_3, t_2) \right) 
	= R_2 \left(t_3, t_1, t_2) \right) \\ 
	= \left(t_3, -\frac{t_{1} t_{2}}{1 + t_{1} + t_1 t_{2}}, -(1 + t_{2} + t_1 t_{2}) \right).
	\end{gathered}
	$$
}
So the identity (8) holds.
\end{proof}

\begin{example}
{\rm
Consider  $w_5 = \sigma_2 \rho_1 \sigma_1 \rho_2 \in FVB_3$. The operator  $F_{FVB}(w_5)$ acts on $ (1, 2, 2)$ in the following way:  
$$
F_{FVB}(w_5) (1, 2, 2)= \left(-5, \frac{4}{11}, -\frac{11}{5} \right) \neq (1, 2, 2).
$$
Therefore, the homomorphism $F_{FVB}$ distinguishes $w_5$ from a trivial braid. 
	}
\end{example}

\end{document}